\documentclass[12pt]{amsart}
\usepackage{amsmath,amssymb,amsbsy,amsfonts,amsthm,latexsym,mathabx,
            amsopn,amstext,amsxtra,euscript,amscd,
            stmaryrd,
            mathrsfs,
            cite,array,mathtools,enumerate}
\usepackage{enumitem}
\usepackage{float} 
\usepackage[english]{babel}
\usepackage{mathtools}
\usepackage{todonotes}
\usepackage{url}
\usepackage[colorlinks,linkcolor=blue,anchorcolor=blue,citecolor=blue,backref=page]{hyperref}
\def\le{\leqslant}
\def\ge{\geqslant}

\usepackage{color}

\usepackage{mathtools}
\usepackage{todonotes}
\usepackage[norefs,nocites]{refcheck}

\renewcommand*{\backref}[1]{}
\renewcommand*{\backrefalt}[4]{%
    \ifcase #1 (Not cited.)%
    \or        (p.\,#2)%
    \else      (pp.\,#2)%
    \fi}

\hypersetup{breaklinks=true}

\usepackage[english]{babel}

\usepackage{mathtools}
\usepackage{todonotes}

\usepackage{enumitem}

\begin{document}

\newtheorem{theorem}{Theorem}
\newtheorem{lemma}[theorem]{Lemma}
\newtheorem{claim}[theorem]{Claim}
\newtheorem{cor}[theorem]{Corollary}
\newtheorem{prop}[theorem]{Proposition}
\newtheorem{definition}{Definition}
\newtheorem{question}[theorem]{Open Question}
\newtheorem{example}[theorem]{Example}
\newtheorem{remark}[theorem]{Remark}

\numberwithin{equation}{section}
\numberwithin{theorem}{section}

 \newcommand{\F}{\mathbb{F}}
\newcommand{\K}{\mathbb{K}}
\newcommand{\D}[1]{D\(#1\)}
\def\scr{\scriptstyle}
\def\\{\cr}
\def\({\left(}
\def\){\right)}
\def\[{\left[}
\def\]{\right]}
\def\<{\langle}
\def\>{\rangle}
\def\fl#1{\left\lfloor#1\right\rfloor}
\def\rf#1{\left\lceil#1\right\rceil}
\def\le{\leqslant}
\def\ge{\geqslant}
\def\eps{\varepsilon}
\def\mand{\qquad\mbox{and}\qquad}

\def\Res{\mathrm{Res}}
\def\vec#1{\mathbf{#1}}

\def \vs {\vec{s}}

\def\bl#1{\begin{color}{blue}#1\end{color}} % color text blue during edits

\newcommand{\Fq}{\mathbb{F}_q}
\newcommand{\Fp}{\mathbb{F}_p}
\newcommand{\Disc}[1]{\mathrm{Disc}\(#1\)}

\newcommand{\Z}{\mathbb{Z}}
\renewcommand{\L}{\mathbb{L}}
%\newcommand{\Nm}[1]{\mathrm{Norm}_{\F{q^k/\Fq}}(#1)}

%%%%%%%%%%%%%%%%%%%%%%%%%
% Alphabet calligraphie %
%%%%%%%%%%%%%%%%%%%%%%%%%
\def\cA{{\mathcal A}}
\def\cB{{\mathcal B}}
\def\cC{{\mathcal C}}
\def\cD{{\mathcal D}}
\def\cE{{\mathcal E}}
\def\cF{{\mathcal F}}
\def\cG{{\mathcal G}}
\def\cH{{\mathcal H}}
\def\cI{{\mathcal I}}
\def\cJ{{\mathcal J}}
\def\cK{{\mathcal K}}
\def\cL{{\mathcal L}}
\def\cM{{\mathcal M}}
\def\cN{{\mathcal N}}
\def\cO{{\mathcal O}}
\def\cP{{\mathcal P}}
\def\cQ{{\mathcal Q}}
\def\cR{{\mathcal R}}
\def\cS{{\mathcal S}}
\def\cT{{\mathcal T}}
\def\cU{{\mathcal U}}
\def\cV{{\mathcal V}}
\def\cW{{\mathcal W}}
\def\cX{{\mathcal X}}
\def\cY{{\mathcal Y}}
\def\cZ{{\mathcal Z}}

\def \brho{\boldsymbol{\rho}}

\def \pf {\mathfrak p}

\def \Prob{{\mathrm {}}}
\def\e{\mathbf{e}}
\def\ep{{\mathbf{\,e}}_p}
\def\epp{{\mathbf{\,e}}_{p^2}}
\def\em{{\mathbf{\,e}}_m}

\def \F{{\mathbb F}}
\def \K{{\mathbb K}}
\def \Z{{\mathbb Z}}
\def \N{{\mathbb N}}
\def \Q{{\mathbb Q}}
\def \R{{\mathbb R}}

\def\GL{\operatorname{GL}}
\def\SL{\operatorname{SL}}
\def\PGL{\operatorname{PGL}}
\def\PSL{\operatorname{PSL}}

\def\ep{\mathbf{e}_p}
\def\eq{\mathbf{e}_q}

\def\Mob{M{\"o}bius }

%
%\title[Sparsity of curves over finite fields]
%{Sparsity of curves over finite fields}

\title[ M{\"o}bius  Orbits at  Prime Times]{On the Dynamical System Generated by the M{\"o}bius  transformation  at 
Prime Times}

 \author[L. M{\'e}rai]{L{\'a}szl{\'o} M{\'e}rai}
\address{L.M.: Johann Radon Institute for Computational and Applied Mathematics, Austrian Academy of Sciences and Institute of Financial Mathematics and Applied Number Theory,
Johannes Kepler University,  Altenberger Stra\ss e 69, A-4040 Linz, Austria} 
\email{laszlo.merai@oeaw.ac.at}
\author[I.~E.~Shparlinski]{Igor E. Shparlinski}
\address{I.E.S.: School of Mathematics and Statistics, University of New South Wales.
Sydney, NSW 2052, Australia}
\email{igor.shparlinski@unsw.edu.au}

\pagenumbering{arabic}

\begin{abstract}
We study the distribution of the sequence of elements of the discrete 
dynamical system 
generated by iterations of  the \Mob map $x \mapsto (ax + b)/(cx+d)$ over a finite 
field of $p$ elements at the moments of time that correspond to prime numbers. 
In particular, we obtain nontrivial estimates of exponential sums with 
such sequences. 
\end{abstract} 


\maketitle

\section{Introduction}
 
\subsection{Motivation and background} 
Let $p$ be a sufficiently large prime and let $\F_p$ be the field of $p$ elements
which we identify with the least residue system modulo $p$, that is, with the set
$\{0, \ldots, p-1\}$.

With any a nonsingular matrix
\begin{equation}
\label{eq:MatrM}
A = \begin{pmatrix} a & b\\ c & d \end{pmatrix} \in \GL_2(\F_p), 
\end{equation}
we consider the \Mob transformation  
$x \mapsto \psi(x)$ associated with $A$  where 
\begin{equation}
\label{Perm}
 \psi(x)  =  \frac{ax + b}{cx + d}. 
\end{equation}

Investigating the distributional properties of elements' orbits  
of the  discrete dynamical system 
$x \mapsto \psi(x)$ on $\F_p$ and of similar systems over residue rings has been a very active area of research, especially in the theory of pseudorandom number 
generators~\cite{GuNiSh, MeSh, NiSh1,NiSh2, NiWi, Shp}, see also~\cite{TopWin, Win} for a general background 
on this and other related  pseudorandom number generators. In fact, in the theory of pseudorandom 
number generators, typically only the special case  $\psi(x) = ax^{-1}  + b$ is considered (which is 
computationally more efficient) but there is no doubt the above results can be extended to any 
transformations of the form~\eqref{Perm}. 

Here we interested in more arithmetic aspects of this problem and  study the distribution of  element in orbits of the \Mob transformation 
at the moments of time that correspond to prime numbers.

More precisely, let $u_0, u_1, \ldots$ be an orbit of the dynamical system 
generated by $\psi$ that originates at some $u_0 \in \F_p$, 
that is, 
\begin{equation}
\label{eq:Gen}
u_{n} =\psi\(u_{n-1}\),
 \qquad n =  1,2,
\ldots\,,
\end{equation}
where $u_0$ is the {\it initial value\/}.

We can also write 
$$
u_n = \psi^{n}(u_0) \qquad n =  1,2,
\ldots\,,
$$
where $\psi^{0}$ is the identity map and $ \psi^{n}$
is the $n$th composition of $\psi$.

Since for any $A \in \GL_2(\F_p)$ the \Mob transformation is reversible, 
it  is obvious that the sequence~\eqref{eq:Gen} is purely
periodic with some period $t \le p$, see~\cite{Chou,FN}
for several results about the possible values of~$t$.
 For example, it is known when such sequences
achieve the
largest possible period, which is obviously $t = p$, see~\cite{FN}.

The series of works~\cite{GuNiSh,NiSh1,NiSh2,Shp} is devoted to the special case of the transformation  $\psi(x) = ax^{-1}  + b$ and several results 
about the distribution of elements of the sequence~\eqref{eq:Gen}
are given. Quite naturally, these results are based on bounds of exponential sums such as
\begin{equation}
\label{eq:Sing Sum}
S_{h}(N) =  \sum_{n=1}^{N} \ep\(hu_n\),
\end{equation}
where for an integer $q$ and a complex $z$ we define
$$
\eq(z) = \exp(2 \pi i z/q).
$$
We also remark that a version of~\cite[Lemma~5.3]{AbSh} improves and generalises the bounds of~\cite{NiSh1} 
on $S_{h}(N)$, see Lemma~\ref{lem:SingSum-1}. It can easily be extended to the multidimensional
settings~\cite{GuNiSh} and thus has direct applications to the theory pseudorandom number generators.

Here, motivated by recent results of Sarnak and Ubis~\cite{SaUb}
on much more complicated dynamical systems on $\mathrm{SL_2(\R)}$, 
we consider the distribution of 
the sequence~\eqref{eq:Gen}
at prime moments of time $n = \ell$. In turn, this is equivalent 
(see~\cite{DrTi}) to studying exponential sums
$$
T_h(N) =  \sum_{\substack{\ell \le N\\\ell~\mathrm{prime}}} \ep\(hu_\ell\).
$$

We also note that the results of~\cite{BCFS,BFGS,Bourg1,GarShp,KMS}
have an interpretation as results on the behaviour at prime
moments of time of the dynamical system generated by the 
linear transformation $x \mapsto gx$ on $\F_p$ (or other residue rings), that is, 
of the sequence $u_0g^\ell$, where $\ell$ runs through the 
primes up to $N$. 

\subsection{Our result and approach} 
Our main result is the following bound:

\begin{theorem}
\label{thm:PrimeSum}  
Assume that the characteristic polynomial of the matrix $A$ given by~\eqref{eq:MatrM} has two distinct roots in $\F_{p^2}$. 
For any $\varepsilon>0$ there exist some $B$ such that 
if the period $t$ of the sequence~\eqref{eq:Gen}  satisfies 
$t  \ge p^{3/4+\varepsilon}$,
then, for a sufficiently large $p$, any constant $C> B$  and  $p^C \ge N \ge p^{B}$we have
$$
\max_{h \in \F_p^*} |T_h(N)| \le N p^{-\eta}, 
$$
where  $\eta> 0$ depends only on $C$ and $\varepsilon$. 
\end{theorem}  

We note that the condition $t  \ge p^{3/4+\varepsilon}$ in Theorem~\ref{thm:PrimeSum} is not very restrictive 
as it is easy to see that  one actually expects $t = p^{1+o(1)}$ for randomly chosen matrix $A \in  \GL_2(\F_p)$,
see also Section~\ref{sec:comm}.  

To establish Theorem~\ref{thm:PrimeSum} we take full advantage of the flexibility of the 
Heath-Brown identity~\cite{HB}, see also~\cite[Proposition~13.3]{IwKow}. In particular it allows us to form 
very skewed bilinear sums to which we can estimate nontrivially. This is crucial for our result since    for sum over essentially 
square regions we do not have nontrivial bounds. 
Another new ingredient is using the Burgess bound
on character sums, see~\cite[Theorem~12.6]{IwKow}, to estimate multiple sums of high dimension. 
Again the agility of  the Heath-Brown identity~\cite{HB} allows us to arrange such sums. 

\subsection{Notation} 
Throughout the paper, the implied constants in 
the symbols `$O$' and `$\ll$'
may occasionally, where obvious, depend on 
the  real positive parameters $C$
and $\varepsilon$, and are absolute otherwise
(we recall that $U \ll V$  is equivalent
to $U = O(V)$).

 \section{Tools from analytic number theory}

\subsection{Products in residue classes} 
  Let $\varphi(k)$ denote the Euler function and let  $\tau(k)$ denote the number of positive integer
divisors of an integer $k \ge 1$. We first  recall the following well-known estimates
\begin{equation}
\label{eq:phitau}
 \tau(k) = k^{o(1)}
 \mand  k \ge  \varphi(k) \gg \frac{k}{\log \log k}
\end{equation}
as $k \to \infty$,
see~\cite[Theorems~317 and~328]{HaWr}.

 First we need to recall the following bound on the distribution of products 
in residue classes. 

Given $\nu \ge 1$ integers $N_1, \ldots, N_\nu\ge 1$ and an arbitrary integer $n$, 
let $R_t(N_1, \ldots, N_\nu; n)$ be the number of solutions to the congruence 
$$
n_1\ldots n_\nu \equiv n \pmod t, \qquad 1 \le n_i \le N_i,  \  i = 1, \ldots, \nu. 
$$

We show that for sufficiently large $N_1, \ldots, N_\nu$, for  $\gcd(n,t)=1$ the value of $N_t(K,M;n)$ 
is close to its expected value. 

\begin{lemma}
\label{lem:UD-Prod}  For any  fixed $\kappa > 0$ there are  some $i_0$ and $\eta > 0$, which depend only on $\kappa$, 
 such that if $\nu >i_0$ then for any integers $t \ge N_1, \ldots, N_\nu \ge t^{1/3+\kappa}$ and $n\geq 1$ with $\gcd(n,t)=1$, 
 we have 
$$
R_t(N_1, \ldots, N_\nu; n)=  \frac{1}{\varphi(t)} N_1^* \ldots N_\nu^* + O\( N_1 \ldots N_\nu t^{-1-\eta}\), 
$$
where
$$
N_i^* = \# \{  1 \le n_i  \le N_i:~\gcd(n_i,t)=1\},  \qquad  i = 1, \ldots, \nu. 
$$
\end{lemma}

\begin{proof}   
Let $\cX_t$ denote the set of multiplicative characters 
modulo $t$ 
and let $\chi_0$ denote the principal character; we refer to~\cite[Chapter~3]{IwKow} for a 
background on multiplicative characters.
We also denote by $\cX_t^* = \cX_t \setminus \{\chi_0\}$ the set of
non principal characters.

Using the orthogonality of multiplicative characters, we can express $R_t(N_1, \ldots, N_\nu; n)$ via
the following character sums, 
\begin{align*}
R_t(N_1, \ldots, N_\nu; n) &= \sum_{n_1=1}^{N_1} \ldots \sum_{n_\nu=1}^{N_\nu} 
\frac{1}{\varphi(t)}\sum_{\chi \in \cX_t} \chi\(n_1\ldots n_\nu  n^{-1}\) \\
& = \frac{1}{\varphi(t)}\sum_{\chi \in \cX_t} \chi\(n^{-1}\)  \prod_{i=1}^\nu 
 \sum_{n_i=1}^{N_i}  \chi\(n_i\).
\end{align*}

We now see that the contribution from the  principal character gives the main term
$N_1^* \ldots N_\nu^*/\varphi(t)$. 

For other characters, since $N_i \ge p^{1/3+\kappa} $, by the Burgess bounds, 
see~\cite[Theorem~12.6]{IwKow},
we see that there is some $\eta> 0$ 
which depends only on $\kappa$ and such that for any $\chi \in  \cX_t^*$ we have 
$$
 \sum_{n_i=1}^{N_i}  \chi\(n_i\) \ll N_i t^{-\eta}, \qquad  i = 1, \ldots, \nu. 
$$ 
Hence for $i_0 =\rf{\eta^{-1}} $ and $\nu > i_0$ we have
$$
 \prod_{i=1}^\nu \sum_{n_i=1}^{N_i}  \chi\(n_i\) 
\ll N_1 \ldots N_\nu t^{-1-\eta}
$$
which concludes the proof.
\end{proof}

\subsection{The Heath-Brown identity} 
\label{sec:H-B Ident} 

As usual, we use $\mu(n)$  to denote the M{\"o}bius function and 
$\Lambda(n)$  to denote the von~Mangoldt function given by
$$
\Lambda(n)=
\begin{cases}
\log \ell &\quad\text{if $n$ is a power of the prime $\ell$,} \\
0&\quad\text{if $n$ is not a prime power.}
\end{cases}
$$

We need the following decomposition of $\Lambda(n)$ which is due to Heath-Brown~\cite{HB}, see also~\cite[Proposition~13.3]{IwKow}.

\begin{lemma} 
\label{lem:H-B ident} 
For any integer $J\geq 1$ and $n<2X$, we have
$$
\Lambda(n)=-\sum_{j=1}^J (-1)^j \binom{J}{j}\sum_{m_1,\ldots, m_j\leq Z}\mu(m_1)\ldots \mu(m_j) \sum_{m_1\ldots m_j n_1\ldots n_j=n}\log n_1,
$$
where $Z=X^{1/J}$.
\end{lemma}

For a real $A>0$ we use $a\sim A$ to denote $A\leq a <2A$. We also write $A \asymp B$ as 
an equivalent of $A\ll B \ll A$.

Hence summing the identity of Lemma~\ref{lem:H-B ident} over all $n \sim N$ and separating the other 
variables in dyadic ranges we obtain

\begin{lemma}
\label{lem:HB sum}
For any integer $J\geq 1$ and arithmetic function $f$, we have
$$
\sum_{n \sim N}\Lambda(n)f(n)\ll \sum_{1\leq j\leq J} %% (\log N)^J 
\left| S_j(\mathbf{M}_j, \mathbf{N}_j)\right| 
$$
for some integer vectors 
$$
(\mathbf{M}_j, \mathbf{N}_j)=(M_{j,1},\ldots, M_{j,j},N_{j,1},\ldots, N_{j,j} )
$$ 
satisfying
$$
M_{j,1},\ldots, M_{j,j} \leq N^{1/J} \mand  M_{j,1}\cdots M_{j,j}N_{j,1}\cdots N_{j,j} \asymp N,
$$
where
\begin{align*}
& S_j\(\mathbf{M}_j, \mathbf{N}_j\)\\
& \qquad =
\sum_{\substack{m_1\ldots m_j n_1\ldots n_j\sim N \\ m_i\sim M_{j,i},\,  n_i \sim N_{j,i}}}\mu(m_1)\ldots \mu(m_j) \log n_1 f(m_1\ldots m_j n_1\ldots n_j), 
\end{align*}  
where $j=1, \ldots, J$, and the implied constants may depend on $J$.
\end{lemma}

 \section{Exponential sums}
\subsection{Single Sums}

We start with  recalling the following variant of~\cite[Lemma~5.4]{AbSh} 
which in particular improves the bound $ S_{h}(N)   \ll N^{1/2} p^{1/4} $ 
of~\cite[Theorem~1]{NiSh1} on the single sums~\eqref{eq:Sing Sum}.

\begin{lemma}\label{lem:SingSum-1} 
Assume that the characteristic polynomial of the matrix $A$ given by~\eqref{eq:MatrM} has two distinct roots in $\F_{p^2}$. 
 Let $t$ be the period of the sequence~\eqref{eq:Gen}.
For any integer numbers $k, N,K \geq 1$, uniformly over $h\in \F_p^*$,  we have
$$
  \sum_{n=K}^{N+K-1} \ep\(hu_{k n}\)  
 \ll  \gcd (k,t)\left(1+\frac{N}{t}\right)p^{1/2} \log p.
$$
\end{lemma}

\begin{proof}  For $N \le t$ this is exactly~\cite[Lemma~5.4]{AbSh}.  Splitting sums of length $N > t$ into $\rf{N/t} \le (1+N/t)$   pieces of length at most $t$, we obtain 
the result. 
\end{proof} 

We now need a similar bound with a co-primality condition.
To simply for the notation we use $\Sigma^\star$ to indicate that the summation is over values 
of the summations variable which are relatively prime to $t$, for example,
$$
\sideset{}{^\star} \sum_{1\le n \le N}f(n) = \sum_{\substack{1\le n \le N\\ \gcd(n,t)=1} } f(n)
$$
for an arithmetic function $f(n)$. 

\begin{lemma}\label{lem:SingSum-Copr} 
Assume that the characteristic polynomial of the matrix $A$ given by~\eqref{eq:MatrM} has two distinct roots in $\F_{p^2}$. 
 Let $t$ be the period of the sequence~\eqref{eq:Gen}.
For any integer numbers $k, N,K \geq 1$, uniformly over $h\in \F_p^*$,  we have
$$
\sideset{}{^\star} \sum_{1\le n \le N} \ep\(hu_{k n}\) 
 \ll  \gcd (k,t)^{1/2} \(N^{1/2} +  N t^{-1/2} \)p^{1/4+o(1)}. $$
\end{lemma}

\begin{proof}  We recall, that $\mu(d)$ denotes the M{\"o}bius function. 
Then using inclusion-exclusion principle we write 
$$
\sideset{}{^\star} \sum_{1\le n \le N} \ep\(hu_{k n}\) = 
\sum_{d \mid t} \mu(d)   \sum_{\substack{1\le n \le N\\   n\equiv 0 \pmod d} } \ep\(hu_{k n}\).
$$
Writing $n = dm$,  we see that by  Lemma~\ref{lem:SingSum-1} each inner sum is bounded 
\begin{equation}
\begin{split}
\label{eq:Bound-1}  
\sum_{\substack{1\le n \le N\\ n\equiv 0 \pmod d} } \ep\(hu_{k n}\) & \ll  \gcd (dk,t)\left(1+\frac{N}{t}\right)p^{1/2} \log p \\
& \le   \gcd (k,t)d\left(1+\frac{N}{t}\right)p^{1/2} \log p .
\end{split}
\end{equation}
It is also trivially bounded by 
\begin{equation}
\label{eq:Bound-2}  
\sum_{\substack{1\le n \le N\\ n\equiv 0 \pmod d} } \ep\(hu_{k n}\) \ll  N/d.
\end{equation}
Multiplying the bounds~\eqref{eq:Bound-1} and~\eqref{eq:Bound-2}, we see that for each $d\mid t$ we have 
$$
\sum_{\substack{1\le n \le N\\ n\equiv 0 \pmod d} } \ep\(hu_{k n}\)  \ll  \sqrt{\gcd (k,t)N \left(1+\frac{N}{t}\right)p^{1/2} \log p}. 
$$
Using the  bound~\eqref{eq:phitau}  on the divisor function $\tau(t) = t^{o(1)}$,   we conclude the proof. 
\end{proof}

\begin{lemma}
\label{lem:SingSum-s}  Assume that the characteristic polynomial of the matrix $A$ given by~\eqref{eq:MatrM} has two distinct roots in $\F_{p^2}$. 
 Let $t$ be the period of the sequence~\eqref{eq:Gen}.
For any integer numbers $N,K \geq 1$, $s\geq 2$ and $M \ge m_{s} > \ldots > m_1 \ge 1$, 
uniformly over $a_1, \ldots, a_s \in \F_p$ not all zeros, we have
$$
 \sum_{n=K}^{N+K-1} \ep\(  a_1  u_{m_1 n}+ \ldots + a_s  u_{m_s n}\) 
 \ll  sM\left(1+\frac{N}{t}\right)p^{1/2} \log p .
$$
\end{lemma} 

\begin{proof}  This bound is a slight generalisation of~\cite[Lemma~5.3]{AbSh}, which corresponds to $s=2$. The general case follows from 
the same arguments without any changes except that we need to establish that the rational function of the form
$$
F(X) = \sum_{j=1}^s    \frac{a_j} {X^{m_j} +\gamma} 
$$
with some $\gamma \in  \F_{p^2}^*$ is non-constant.  Without loss of generality, we can assume that $a_1\ne 0$. Then the desired property of $F$  is obvious from examining the
leading term $a_1X^{m_s+\ldots + m_2}$  of the numerator. \end{proof}

\begin{remark}
We remark that for full sums, that is, for $N=t$ the logarithmic term $\log p$ is not needed so in both Lemma~\ref{lem:SingSum-1} and~\ref{lem:SingSum-s}  the term 
$\left(1+N/t\right)p^{1/2} \log p$ can be replaced with $p^{1/2} \log p + Nt^{-1}p^{1/2}$. This however does not affect the final result 
\end{remark}

\subsection{Multiple Sums}

Next we need to estimate certain multiple sums.

\begin{lemma}
\label{lem:MultExp}  Assume that the characteristic polynomial of the matrix $A$ given by~\eqref{eq:MatrM} has two distinct roots in $\F_{p^2}$.  
Let $t$ be the period of the sequence~\eqref{eq:Gen}.  
For any  fixed $\kappa > 0$ there are  some $j_0$ and $\zeta > 0$, which depend only on $\kappa$, 
 such that if 
$$
\nu >  j_0  \mand t \ge p^{1/2 + \kappa}, 
$$ 
then for any  integers $t \ge N_1, \ldots, N_\nu \ge t^{1/3+\kappa} $  and  
 $h\ge 1$,  uniformly over $h \in \F_p^*$, we have
$$  \sideset{}{^\star} \sum_{1 \le n_1 \le N_1} \ldots  \sideset{}{^\star} \sum_{1 \le n_\nu \le N_\nu} 
\,\ep(h u_{k n_1 \ldots n_\nu})  \ll \gcd(k,t) N_1 \ldots N_\nu t^{-1-\zeta}. 
$$
\end{lemma}

\begin{proof}    By Lemma~\ref{lem:UD-Prod}   we have 
\begin{align*}
 \sideset{}{^\star} \sum_{1 \le n_1 \le N_1}& \ldots  \sideset{}{^\star} \sum_{1 \le n_\nu \le N_\nu} 
\,\ep(h u_{k n_1 \ldots n_\nu}) \\
 & =  \sum_{\substack{n=1\\ \gcd(n,t)=1}}^t R_t\(N_1, \ldots , N_\nu;n\) \,\ep\(hu_{k   n}\)\\
& =  \sum_{\substack{n=1\\ \gcd(n,t)=1}}^t  \(\frac{1}{\varphi(t)} N_1 \ldots N_\nu  + O\( N_1 \ldots N_\nu  t^{-1-\eta}\)\)  \ep\(h u_{ku n}\)\\
& =  \frac{1}{\varphi(t)} N_1 \ldots N_\nu   \sum_{\substack{n=1\\ \gcd(n,t)=1}}^t    \ep\(h u_{k n}\)  
+ O\( N_1 \ldots N_\nu t^{- \eta}\), 
\end{align*}  
provided that $\nu > j_0$, where  $j_0$ and  $ \eta>0$  depend  only on $\kappa$.

Now using Lemma~\ref{lem:SingSum-Copr} with $N=t$ and then recalling that by~\eqref{eq:phitau} we have $\varphi(t) = t^{1+o(1)}$,  we obtain 
\begin{align*}
 \sideset{}{^\star} \sum_{1 \le n_1 \le N_1}& \ldots  \sideset{}{^\star} \sum_{1 \le n_\nu \le N_\nu} 
\,\ep(h u_{k n_1 \ldots n_\nu}) \\ &  \ll  \frac{1}{\varphi(t)} N_1 \ldots N_\nu    \gcd(k,t)^{1/2}   t^{1/2} p^{1/4+o(1)} +  N_1 \ldots N_\nu t^{-\eta}\\
  &  \ll  N_1 \ldots N_\nu \gcd(h,t)^{1/2}   t^{-1/2} p^{1/4+o(1)} + N_1 \ldots N_\nu t^{- \eta}, 
  \end{align*}
and the desired result follows.\end{proof}

\subsection{Bilinear Sums}

Next we need to estimate certain bilinear sum.

\begin{lemma}
\label{lem:BilinExp} Assume that the characteristic polynomial of the matrix $A$ given by~\eqref{eq:MatrM} has two distinct roots in $\F_{p^2}$. 
 Let $t$ be the period of the sequence~\eqref{eq:Gen}.
For any  positive integers $M, K \ge 1$ and  any two sequences $\boldsymbol{\alpha}=(\alpha_k)_{k=1}^K$
and $\boldsymbol{\beta}=(\beta_m)_{m=1}^M$ of complex numbers, uniformly over $h\in \F_p^*$,  we have
\begin{align*}
& \left| \sum_{k=1}^K \sum_{m=1}^M
\alpha_k\,\beta_m \,\ep(h u_{k m})\right| \\
&\quad  \le
\|\boldsymbol{\alpha} \|_\infty \|\boldsymbol{\beta} \|_\infty KM   \(M^{-1/2} + K^{-1/2}M^{1/2} p^{1/4} + M^{1/2} p^{1/4} t^{-1/2}    \)p^{o(1)}, 
\end{align*}
where
$$
\|\boldsymbol{\alpha} \|_\infty=\max_{k\le K}|\alpha_k|\mand \|\boldsymbol{\beta} \|_\infty=\max_{m\le M}|\beta_m|.
$$
\end{lemma}  
  
\begin{proof}  We have
\begin{equation}
\label{eq:A W}
\left| \sum_{k=1}^K\sum_{m=1}^M
\alpha_k\,\beta_m \,\ep(h u_{k m})\right|\le \|\boldsymbol{\alpha} \|_\infty W, 
\end{equation}
where
$$
W = \sum_{k=1}^K\left|\sum_{m=1}^M \beta_m \,\ep(h u_{k m})\right|.
$$
Using the Cauchy 
inequality, we derive
\begin{align*}
W^{2} & \le K   \sum_{k=1}^K\left| \sum_{m=1}^M
  \beta_m \ep(h u_{k m})\right|^{2} \\
& \le \|\boldsymbol{\beta} \|_\infty^2 K  \sum_{m, n=1}^M  
\left|  \sum_{k=1}^K  \ep\(a\(u_{km} - u_{kn}\)\)\right|. 
\end{align*}

We now use the trivial bound $K$ for $M$ choice $m = n$  and use Lemma~\ref{lem:SingSum-s} for 
the remaining values. Hence, we derive 
$$
W^2  \ll \|\boldsymbol{\beta} \|_\infty^2K  \(M  K + M^3  (1+ K/t) p^{1/2} \log p\).
$$
Substituting this bound in~\eqref{eq:A W},  after simple calculations we conclude the  proof.  
\end{proof}

\section{Proof of Theorem~\ref{thm:PrimeSum}}

\subsection{Preliminaries} 
We recall  the notation  $a\sim A$ and $A \asymp B$  from Section~\ref{sec:H-B Ident}. 

Using partial summation we see that  instead of   $T_h(N)$ it is enough to estimate  the sums
$$
U_h(N)=\sum_{n \le  N} \Lambda(n) \ep(hu_n), 
$$
which again via partial summation, and discarding the contribution of order $N^{o(1)}$ from primes $\ell \mid t$, can be reduced to the sums
$$
V_h(N)= \sideset{}{^\star}  \sum_{n \sim N} \Lambda(n) \ep(hu_n)
$$
over dyadic intervals.

To estimate $V_h(N)$, we   fix some $\kappa>0$ and define an integer   $J \ge 2$ by the 
inequalities 
 \begin{equation} \label{eq:def J}
N^{1/J}\le  t p^{-1/2+\kappa/2} \le N^{1/(J-1)} .
\end{equation}  
Since $N \ge p^B$ 
we see that if $B$ is large enough then  so is $J$ (in particular $J\ge 2$). 
Furthermore, since $t \ge p^{3/4 + \kappa} \ge p^{3/4}$,  
 \begin{equation} \label{eq:NJ LB}
N^{1/J} =  N^{(J-1)/J(J-1)} \ge   \( t p^{-1/2-\kappa/2} \)^{(J-1)/J} \ge t p^{-1/2-\kappa}, 
\end{equation}  
provided that $B$ and thus $J$ are large enough in terms of $\kappa$.

We also define $j_0$ be as in Lemma~\ref{lem:MultExp} in therms of $\kappa$.

 Next we note that by Lemma~\ref{lem:HB sum}  it is enough to estimate the sums 
of the form
$$
S(\mathbf{M}, \mathbf{N})= \sideset{}{^\star} 
\sum_{\substack{m_1\ldots m_j n_1\ldots n_j\sim N \\ m_i\sim M_{i},\,  n_i \sim N_{i}}}\mu(m_1)\ldots \mu(m_j) \log n_1 \ep\(au_{m_1\ldots m_j n_1\ldots n_j}\)
$$
for some integer vectors $(\mathbf{M},  \mathbf{N})=(M_{1},\ldots, M_{j},N_{1},\ldots, N_{j} )$  with $j \le J$ and satisfying
$$
M_{1},\ldots, M_{j} \leq N^{1/J}\mand  Q \asymp N,
$$
where 
$$
Q = 2^{2j} \prod_{i=1}^j M_i \prod_{i =1}^jN_i.
$$

In particular, it is enough to prove that under the above conditions we have 
 \begin{equation} \label{eq:Bound S}
S(\mathbf{M}, \mathbf{N}) \ll Np^{-\rho}
\end{equation}  
for some   $\rho>0$ which depends only on $\varepsilon$. 

% We also fix some $\kappa > 0$ and let $j_0$ be as in Lemma~\ref{lem:MultExp}.

\subsection{Large values of  the product $M_1 \ldots M_j$} 
Assume that 
$$M_1 \ldots M_j\ge N^\kappa.
$$
There for some $i$ we have $M_i \ge N^{\kappa/j} \ge N^{\kappa/J}$.
Hence we apply Lemma~\ref{lem:BilinExp}  with $M = M_i$ and $K = Q/M_i$
and 
$$\|\boldsymbol{\alpha} \|_\infty \le 1 \mand  \|\boldsymbol{\beta} \|_\infty  \le \log Q = N^{o(1)}.
$$
Therefore, we obtain 
\begin{align*}
 S(\mathbf{M}, \mathbf{N}) & \ll
Q  \(M_i^{-1/2} + ( Q/M_i)^{-1/2}M^{1/2} p^{1/4} + M_i^{1/2} p^{1/4} t^{-1/2}    \)N^{o(1)}\\
&  \ll
N  \(  M_i^{-1/2} + M_i N^{-1/2} p^{1/4} + M_i^{1/2} p^{1/4} t^{-1/2}    \)N^{o(1)}.
%%N^{1\kappa/2J}^{-1/2}
\end{align*}

Since by our choice of parameters~\eqref{eq:def J} we have 
$$
t \ge N^{1/J} p^{1/2+\kappa} \ge M_i   p^{1/2+\kappa}  
$$
as well as 
$$
M_i \ge  N^{\kappa/J}  \mand  N \ge    N^{2/J} p^{1/2+\kappa} \ge M_i^2   p^{1/2+\kappa}
$$
(provided that $B$ and thus $J$ are large enough), 
we obtain a bound of the desired type~\eqref{eq:Bound S}.

\subsection{Large values among $N_1, \ldots, N_j$} 
If $N_k\ge t$ for some $k=1, \ldots, j$, then by Lemma~\ref{lem:SingSum-1}, applied to each 
of the 
$$O(M_1\ldots M_jN_1\ldots N_j/N_k)  = O(Q/N_k)
$$ 
sums over $n_k$ (and using partial summation to eliminate the effect of 
$\log n_1$ if $k=1$)  we obtain 
\begin{align*}
\left |S(\mathbf{M}, \mathbf{N})\right | &\le p^{1/2+o(1) }\left(1+\frac{N_k}{t}\right)Q/N_k\\
&\leq p^{1/2+o(1) } t^{-1} Q  \le N  p^{-1/4}. 
\end{align*}

\subsection{Remaining cases} 
So we can now assume that 
 \begin{equation}\label{eq:Asump}
 M_1 \ldots M_j\le N^\kappa \mand  N_1,   \ldots, N_j <  t.
\end{equation}
First we note that if there is a set $\cI \subseteq \{1, \ldots, j\}$ 
of cardinality $\#\cI >j_0$ (we recall that $j_0$ is chosen as in Lemma~\ref{lem:MultExp}) and such that 
$$
N_i \ge t^{1/3+\kappa}, \qquad i \in \cI,
$$
then using  Lemma~\ref{lem:MultExp} instead of Lemma~\ref{lem:SingSum-1}, as in the above we obtain the desired 
bound~\eqref{eq:Bound S}.

Otherwise, that is, if no more than $j_0$ elements among $N_1,   \ldots, N_j$ exceed or equal $t^{1/3+\kappa}$, using~\eqref{eq:Asump}, we obtain 
$$
N^{1-\kappa} \ll  N_1\ldots N_j \le  \(t^{1/3+\kappa}\)^{j-j_0}  t^{j_0}. 
$$
From which we trivially derive 
$$
N^{1-\kappa} \ll  N_1\ldots N_j \le  \(t^{1/3+\kappa}\)^{J}  t^{2j_0/3}
$$
and hence 
 \begin{equation} \label{eq:NJ UB}
N^{1/J} \le t^{(1/3+\kappa)/(1-\kappa)+2j_0/3J(1-\kappa)} <  t^{1/3+4\kappa}, 
\end{equation}  
provided that 
$$
\kappa \le 1/3 \mand J\ge \kappa^{-1}(1-\kappa)^{-1} j_0, 
$$
which holds  if $B$ is large enough in terms of $\kappa$ and $j_0$, and 
thus in terms of  $\kappa$  and $\varepsilon$.  
Combining~\eqref{eq:NJ LB} and~\eqref{eq:NJ UB} we derive
$$
 t p^{-1/2-\kappa}< t^{1/3+4\kappa},
 $$
 which contradicts the assumption $t \ge p^{3/4+\varepsilon}$, 
 provided 
 $$
 \frac{1/2 +\kappa}{2/3-4\kappa}=   \frac{3 +6\kappa}{4-24\kappa}   < 3/4 + \varepsilon
 $$
 which holds for $\kappa < \varepsilon/50$. 
 Hence,  taking $\kappa = \varepsilon/51$ we conclude the proof.

\section{Remarks}
\label{sec:comm}

It is well-known 
that under the Generalised Riemann Hypothesis,  uniformly over $\chi \in \cX_t^*$, 
we have the bound 
$$
\left|  \sum_{x=1}^N  \chi\(x\)\right| \le N^{1/2}t^{o(1)} ,
$$
see~\cite[Section~1]{MoVau}; 
it can also be  derived   from~\cite[Theorem~2]{GrSo}. This allows us to replace in 
 Lemma~\ref{lem:UD-Prod}   the condition  $N_1, \ldots, N_\nu \ge t^{1/3+\kappa}$
 with $N_1, \ldots, N_\nu \ge t^{\kappa}$. In turn, the inequality~\eqref{eq:NJ UB}  becomes 
$$
N^{1/J} \le t^{\kappa/(1-\kappa)+j_0/J(1-\kappa)}   <  t^{3\kappa}, 
$$
provided
$$
\kappa \le 1/2 \mand J\ge \kappa^{-1}(1-\kappa)^{-1} j_0.
$$
Hence, we now easily see that under the  Generalised Riemann Hypothesis the 
result of Theorem~\ref{thm:PrimeSum} holds already for $t  \ge p^{1/2+\varepsilon}$.

Our method also works for exponential sums 
with the sequence~\eqref{eq:Gen} twisted with the M{\"o}bius function:
$$
r_h(N) = \sum_{n \le N} \mu(n)\ep\(hu_n\), \qquad h \in\Z.
$$
In fact, such sums have been estimated in~\cite{AbSh} with some logarithmic saving for rather small values of $N$.
The method and results of this work, such as Lemmas~\ref{lem:MultExp}  and~\ref{lem:BilinExp}, apply to 
longer sums, but yield a power saving. 

One can also use our approach to estimate exponential sum with orbits of \Mob transformation along sequences 
with other arithmetic conditions. For example, using a  combinatorial  identity of Vaughan~\cite[Lemma~10.1]{Vau} 
one can relate the sums over {\it smooth numbers\/} (that numbers without large prime divisors)  to double sums 
and then use our results such as Lemma~\ref{lem:BilinExp}

Probably the most challenging open question here is to obtain nontrivial 
results in the case of the period $t < p^{1/2}$. We note that the striking results and method of 
Bourgain~\cite{Bourg1,Bourg2} do not seem to apply even to the 
case of the sums~\eqref{eq:Sing Sum} over consecutive elements.

 \section*{Acknowledgement}
 
 The author is also very grateful to Alina Ostafe for 
finding various imprecisions in the initial version of this paper 
and useful comments.
 
During the preparation of this work L.M. was supported by the Austrian Science Fund
Project P31762, and I.S. was supported by the  Australian Research Council  Grants DP170100786 and DP180100201
and by the Natural Science Foundation of China Grant~11871317.

\end{document}